\documentclass{birkjour}

\usepackage{amssymb, amsmath, amsthm, amsfonts, amscd}
\usepackage{xcolor}
\usepackage[sort, numbers]{natbib}
\usepackage{datetime}
\usepackage{hyperref}
\hypersetup{
	colorlinks=true,
	linkcolor=cyan,
	filecolor=magenta,
	urlcolor=cyan,
	citecolor= cyan,
}
\usepackage{comment}
\usepackage{mathabx}
\usepackage{enumerate}
\usepackage[mathscr]{euscript}
\usepackage{enumitem}
\usepackage{commath}
\usepackage[capitalize,nameinlink]{cleveref}
\newtheorem{theorem}{Theorem}[section]
\newtheorem*{theorem*}{Theorem}
\newtheorem{lemma}[theorem]{Lemma}
\newtheorem{corollary}[theorem]{Corollary}
\theoremstyle{definition}
\newtheorem{definition}[theorem]{Definition}

\theoremstyle{remark}
\newtheorem{remark}[theorem]{Remark}

\numberwithin{equation}{section}

\begin{document}
\title[The Weyl--von Neumann theorem ]{The Weyl--von Neumann theorem for antilinear skew-self-adjoint operators}
	
\author{G. Ramesh}
\address{Department of Mathematics, Indian Institute of Technology Hyderabad, Kandi, Sangareddy, Telangana, India 502 284.}
\email{rameshg@math.iith.ac.in}



\begin{abstract}
 In this article, we prove the Weyl--von Neumann theorem for bounded antilinear skew-self-adjoint operators. More specifically, we prove the following:\\
 Let $A$ be a  bounded antilinear skew-self-adjoint operator on a separable Hilbert space $H$ whose kernel is either even dimensional or infinite dimensional.  Let $1<p<\infty$.  Then for every $\epsilon>0$ there exists an antilinear  block skew-diagonal  operator $D$ and an antilinear skew-self-adjoint  Schatten $p$-class operator $K$ such that $A=K+D$ with $\|K\|_{p}<\epsilon$. 
 
 As a consequence, we prove the Weyl--von Neumann theorem for complex skew-symmetric operators:\\ 

 Let $\tau$ be a conjugation on $H$ and let $T$ be a  bounded linear operator $\tau$-skew-symmetric with $\dim N(T)=\infty$ or $\dim N(T)$ is even.  Let  $1<p<\infty$. Then for every $\epsilon>0$, there exists a $\tau$-skew-symmetric Schatten $p$-class operator $K$, a skew-symmetric block  diagonal operator $D$  and a unitary operator $U$ such that $T=K+UDU^{tr}$ and $\|K\|_{p}<\epsilon$, where $U^{tr}$ is the transpose of $U$ with respect to an orthonormal basis ${\{e_n:n\in \mathbb N}\}$ such that $\tau(e_n)=e_n$ for each $n\in \mathbb N$.

 Furthermore, the above result holds even without any assumption on the dimension of 
$N(T)$, provided that $N(T)=N(T^*)$. 
 \end{abstract}
\maketitle
\section{Introduction and preliminaries}
\subsection{Introduction}
In this article, we obtain a Weyl--von Neumann theorem for antilinear skew-self-adjoint bounded operators. An analogue of this theorem for antilinear self-adjoint bounded operators was proved by Ruotsalainen in \cite{Santtu}.

The classical Weyl--von Neumann theorem, due to Weyl \cite{Weyl1909}, asserts that every bounded self-adjoint operator on a separable Hilbert space can be written as the sum of a diagonal operator and a compact operator of arbitrarily small norm. von--Neumann \cite{vonNeumann1935} later sharpened this by showing that the compact operator can be replaced by a Hilbert--Schmidt operator of arbitrarily small Hilbert--Schmidt norm. This result was substantially refined by Kuroda \cite{Kuroda1958}, who showed that the perturbing operator can in fact be taken in any Schatten class $\mathcal K_{p}(H)$ with $1<p<\infty$. Berg \cite{Berg1971} later extended the Weyl--von Neumann theorem to unbounded normal operators.

The present result is also closely related to the quaternionic Weyl--von Neumann--Berg theorem for quaternionic normal operators established by the author in \cite{Ramesh2016}. Writing the polar decomposition $A=\kappa|A|$, skew-self-adjointness of $A$ forces its modulus $|A|$ to commute with the antiunitary factor $\kappa$, which is exactly the condition making $|A|$ quaternionic-linear with respect to the structure $(H,\kappa)$ studied in \cite{Ramesh2016}. For self-adjoint operators, the earlier result gives, in this quaternionic setting, the compact perturbation result and, through the rectifiable-spectrum result, a Hilbert--Schmidt type approximation. In contrast, the present argument yields an antilinear skew-self-adjoint perturbation in every Schatten class $\mathcal K_p^a(H)$, $1<p<\infty$, and in particular gives the stronger Schatten-$p$ approximation for $1<p<2$.

For skew-symmetric operators, a Weyl--von Neumann type theorem was proved by Bu and Zhu in \cite{SZhu}: for a self-adjoint operator $T$ satisfying $CTC=-T$ with respect to a conjugation $C$, they obtain, for every $1<p<\infty$ and every $\epsilon>0$, a self-adjoint diagonal operator $D$ with $CDC=-D$ such that $\|T-D\|_p<\epsilon$. Our setting is different: a $\tau$-skew-symmetric operator $T$, defined by $\tau T\tau=-T^*$, satisfies $T^{\mathrm{tr}}=-T$ with respect to the transpose induced by $\tau$, and is thus genuinely antisymmetric rather than self-adjoint. Such an operator does not, in general, admit a diagonal form compatible with this symmetry; consequently, a block-skew-diagonal form together with a conjugating unitary $U$, as in Theorem~\ref{skewsymWvN}, is the natural analogue of Bu--Zhu's result for this symmetry class. Kuroda's refinement remains the relevant classical antecedent in the sense that both results place the perturbation in every Schatten class $\mathcal K_p(H)$, $1<p<\infty$, rather than only achieving a Hilbert--Schmidt or compact bound.

Antilinearity is a natural phenomenon in bipartite quantum systems. It is well known that there is a one-to-one correspondence between vectors and certain antilinear maps, called EPR-maps (see \cite{Uhlmann} for more details). Antilinear operators have applications in the study of quantum entanglement \cite{ArensVaradarajan} and quantum teleportation \cite{Kuruczetal}. These maps play an important role in the Tomita-Takesaki theory. A real linear operator on a Hilbert space can be written as a sum of a complex linear operator and an antilinear operator \cite{Santtu2}.

The development of the theory for antilinear operators on Hilbert spaces is rather limited in comparison to the theory of linear operators. This has made it difficult to develop the theory in a systematic manner. Nevertheless, in certain cases, these operators can be connected; for example, antilinear normal operators with conjugate normal operators, self-adjoint antilinear operators with complex symmetric operators, and skew-self-adjoint operators with complex skew-symmetric operators.

In this article, we prove that an antilinear skew-self-adjoint operator can be perturbed by an antilinear skew-self-adjoint Schatten $p$-class operator with arbitrarily small Schatten $p$-norm, $1<p<\infty$, to get an antilinear block skew-diagonal operator. Precisely, we prove the following theorem.

\begin{theorem}\label{WvN}
Let $A$ be a bounded skew-self-adjoint antilinear operator on a separable Hilbert space $H$ such that either $\dim N(A)=\infty$ or $\dim N(A)$ is even. Then, for every $\epsilon>0$ there exists an \textbf{antilinear block skew-diagonal operator} $D$ and an antilinear skew-self-adjoint Schatten $p$-class operator $K$ such that

$$
A=K+D\quad\text{and}\quad \|K\|_{p}<\epsilon,
$$

where $1<p<\infty$.
\end{theorem}

\begin{definition}
We call a  bounded antilinear operator $D$ on $H$ a \emph{block skew-diagonal operator} if there exist a sequence $(d_n)$ of nonnegative real numbers and an orthonormal basis
${\{e_n,f_n:n\in\mathbb N}\}$ such that
\begin{equation*}
De_n=d_nf_n,\qquad Df_n=-d_ne_n,\qquad n\in\mathbb N.
\end{equation*}
\end{definition}

Equivalently, with respect to the orthonormal basis
${\{e_1,f_1,e_2,f_2,\ldots}\}$, we have

$$
D=\bigoplus_{n=1}^{\infty}d_n\kappa,
$$

where $\kappa:\mathbb C^2\to\mathbb C^2$ is the anticonjugation defined by
\begin{equation*}
\kappa(x,y)=(\overline y,-\overline x),\qquad \text{for all }(x,y)\in\mathbb C^2.
\end{equation*}

Note that if $D$ is a block skew-diagonal operator, then $D$ is skew-self-adjoint and $D^2$ is a linear diagonal operator with respect to the same orthonormal basis. More precisely, we have
\begin{equation*}
D^2e_n=-d_n^2e_n,\qquad
D^2f_n=-d_n^2f_n,\qquad n\in\mathbb N.
\end{equation*}

Thus, the diagonal entries of $D^2$ are

$$
-d_1^2,-d_1^2,-d_2^2,-d_2^2,\ldots.
$$

In particular,

$$
D^2=\bigoplus_{n=1}^{\infty}(-d_n^2I_2).
$$

Since the spectrum $\sigma(T)$ of an antilinear skew-self-adjoint operator $T$ is a subset of ${\{0}\}$ \cite[Proposition 2.16]{Santtu2}, we cannot get a diagonal operator in Theorem \ref{WvN}. But we get a block skew-diagonal operator as given in Theorem \ref{skewdiagonalization}.

We also prove a similar theorem for complex skew-symmetric operators, which can be stated as follows.

\begin{theorem}\label{skewsymWvN}
Let $\tau$ be a conjugation on $H$ and let $T$ be a $\tau$-skew-symmetric operator with $\dim N(T)$ even or $\dim N(T)=\infty$. Then for every $\epsilon>0$, there exist a $\tau$-skew-symmetric Schatten $p$-class operator $K$, a skew-symmetric operator $D$, and a unitary operator $U$ such that
\begin{equation*}
T=K+UDU^{tr}\quad\text{and}\quad \|K\|_{p}<\epsilon,
\end{equation*}
where $1<p<\infty$ and $A^{tr}$ denotes the transpose of the operator $A$ with respect to an orthonormal basis ${\{e_n:n\in\mathbb N}\}$ of $H$ such that $\tau(e_n)=e_n$ for all $n\in\mathbb N$.
\end{theorem}

   \subsection{Preliminaries} Throughout the article, $H$ denotes a complex, infinite dimensional, separable Hilbert space, and $\mathcal B(H)$ denotes the space of all bounded linear operators on $H$. We use the convention that the inner product is linear in the first variable.
   
   For $T\in\mathcal B(H)$, we denote the kernel and range of $T$ by $N(T)=\{x\in H:Tx=0\}$ and  $R(T)=\{Tx:x\in H\}$, respectively.

The adjoint of $T\in \mathcal B(H)$ will be denoted by  $T^*$, which satisfies 
\begin{equation*}
\langle Tx,y\rangle=\langle x,T^*y\rangle,\; \text{for all}\; x,y\in H.
\end{equation*}

We say that $T$ is self-adjoint if $T=T^*$, skew-self-adjoint if $T^*=-T$, normal if $T^*T=TT^*$ and unitary if $T^* T=TT^*=I$. We say that $T$ is an isometry if $\left\|Tx\right\|=\left\|x\right\|$ for all $x\in H$ and a partial isometry if $\left\|T x\right\|=\left\|x\right\|$ for all $x\in N(T)^{\perp}$.

If $T$ is self-adjoint and $\langle Tx, x\rangle \geq 0$ for every $x\in H$, we say $T$ is positive.

For $T\in \mathcal{B}(H)$, the polar decomposition of $T$ is given by
\begin{equation}\label{polardecomposition}
T = U|T |,
\end{equation}
where $U$ is a partial isometry with $N(U) =N(T)$ and $|T|=(T^*T)^{\frac{1}{2}}$, the positive square root of $T^*T$.

An operator $A:H\to H$ is said to be \emph{antilinear} if
$$A(x+y)=Ax+Ay
\quad\text{and}\quad
A(\alpha x)=\overline{\alpha}\,Ax
$$

for all $x,y\in H$ and $\alpha\in\mathbb C$. 

We denote the set of all antilinear operators on $H$ by $\mathcal B_a(H)$, and the set of all compact antilinear operators on $H$ by $\mathcal K_a(H)$.

\begin{definition}\cite{ptak}\label{antilineardefn}
For a bounded antilinear operator $A$ on $H$, there exists a unique antilinear operator $A^{\#}$, called the \emph{antilinear adjoint} of $A$, satisfying
\begin{equation}
\langle Ax,y\rangle=\overline{\langle x,A^{\#}y\rangle}, \;\text{for all}\; x,y\in H.
\end{equation}
\end{definition}

\begin{definition}
An antilinear operator $A$ on $H$ is called
\begin{enumerate}
\item \emph{antilinear self-adjoint} if $A^{\#}=A$;
\item \emph{antilinear skew-self-adjoint} if $A^{\#}=-A$;
\item \emph{antilinear normal} if $A^{\#}A=AA^{\#}$;
\item \emph{antilinear unitary} if $A^{\#}A=AA^{\#}=I$.
\end{enumerate}
\end{definition}

\begin{definition}
	 A map $\tau: H\rightarrow H$ is said to be a conjugation if it satisfies the following:
	\begin{enumerate}[label=(\roman*)]
		\item $\tau$ is antilinear, i.e., $\tau(\alpha x +y)=\bar{\alpha}\tau(x)+\tau(y)$ for $x,y\in H$ and $\alpha\in \mathbb{C}$,
		\item  $\tau^2=I$,
		\item  $\langle \tau x,\tau y\rangle=\langle y,x\rangle$, for all $x,y\in H$.
		
	\end{enumerate}
	That is, $\tau$ is antilinear isometric involution.
\end{definition}
It is well known that given an orthonormal basis ${\{e_n:n\in \mathbb{N}}\}$ of $H$, there exists a conjugation $\tau$ on $H$ such that $\tau(e_n)=e_n$ for all $n\in \mathbb{N}$. That is, 
\begin{equation}\label{conjugationeq}
    \tau(x)=\displaystyle \sum_{n=1}^{\infty}\overline{\langle x,e_n\rangle} e_n,\; \text{for all}\; x\in H.
\end{equation}
On the other hand, if $\tau$ is a conjugation on $H$, then there exists an orthonormal basis ${\{e_n:n\in \mathbb{N}}\}$ such that $\tau(e_n)=e_n$ for all $n\in \mathbb{N}$. That is, $\tau$ is given by equation \ref{conjugationeq}.

Let $J\in\mathcal B_a(H)$. We say that $J$ is a \emph{partial anticonjugation} if $N(J)^\perp$ is invariant under $J$ and the restriction
$J|_{N(J)^\perp}$ is an anticonjugation on $N(J)^\perp$. In particular, the linear operator $J^2=-P_{N(J)^\perp}$, where $P_{N(J)^\perp}$ is the orthogonal projection onto $N(J)^\perp$.

Consequently, the partial anticonjugation $J$ can be extended to an anticonjugation $\widetilde J=J\oplus J'$ on the entire space $H$, where $J'$ is an anticonjugation on $N(J)$ (see \cite[Section 2.2]{garcia2} and \cite{LiZhou2014} for more details).
\begin{definition}\cite{garcia1}
	Let $T\in \mathcal{B}(H)$ and $\tau$ be a conjugation on $H$. Then $T$ is said to be  $\tau$-symmetric if $T^*=\tau T \tau$. Equivalently, $T^*\tau=\tau T$, or $T\tau=\tau T^*$.
	
	A bounded linear operator $T$ is said to be complex symmetric if there exists a conjugation $\tau$ with respect to which it is $\tau$-symmetric. We refer to \cite{garcia1,garcia2} for more details about complex symmetric operators.
\end{definition}
\begin{definition}\cite{LiZhou2014}
	Let $\tau$ be a conjugation on $H$ and $T\in \mathcal{B}(H)$. Then $T$ is said to be  $\tau$-skew-symmetric if $\tau T^*\tau=-T$.
\end{definition}
An operator $T\in \mathcal{B}(H)$ is said to be skew-symmetric if there exists a conjugation $\tau$ on $H$ such that $T$ is $\tau$-skew-symmetric.

\begin{definition}\cite{LiZhou2014}
	An anticonjugation on $H$ is an antilinear map $k:H\rightarrow H$ satisfying $k^2=-I$ and $\langle kx,ky\rangle=\langle y,x\rangle$ for all $x,y\in H$.
\end{definition}
 A characterization of anticonjugations on a Hilbert space can be given as follows:
\begin{lemma}\cite[Lemma 2.2]{LiZhou2014}
Let $\kappa$ be an isometric antilinear map on $H$. Then $\kappa$ is an anticonjugation
if and only if there exists an  orthonormal basis ${\{e_n, f_n:n\in \mathbb{N} }\}$ such that 
\begin{equation*}
\kappa e_n= f_n\;\; \text{and}\;\; \kappa f_n = -e_n\; \text{for each}\;n\in  \mathbb{N}.
\end{equation*}
\end{lemma}
\section{Main Results}

In this section, we prove our main results.  We accomplish this by establishing a factorization theorem for antilinear skew-self-adjoint operators. This can be compared to \cite[Theorem 3.16]{RSVpolar}.
\begin{theorem}\label{polardecomp}
   Let $A\in \mathcal B_{a}(H)$. Then the following are equivalent.
   \begin{enumerate}
   \item\label{sksa} $A$ is skew-self-adjoint with $\dim N(A)=\infty$ or $\dim N(A)$ is even
\item\label{factorization} There exists an anticonjugation $k$ on $H$ such that
\begin{equation*}
A=k|A|\; \text{with} \;k|A|=|A|k,
\end{equation*}
where $|A|$ is the unique positive square root of the linear operator $A^{\#}A$.
\end{enumerate}
\end{theorem}
\begin{proof}
First assume that $A^{\#}=-A$. Let $\tau$ be a conjugation on $H$. Then $T=\tau A$ is a $\tau$-skew-symmetric operator. By \cite[Theorem 3.8]{RSVpolar}, there exists a unique  partial anticonjugation $J$ on $H$ such that $\tau A=\tau J|T|$ and $J|T|=|T|J$.  Note that $|T|=|A|$. Hence, we have $A=J|A|=|A|J$. Since $N(J)=N(|T|)=N(|A|)=N(A)$, it is clear that $N(A)$ is invariant under $A$. Since $A^{\#}=-A$, it follows that $N(A)$ is invariant under both $A$ and $A^{\#}$, therefore it is a reducing subspace for $A$, so we can write $J=0\oplus \widetilde{J}$, where $\widetilde{J}=J|_{N(A)^{\bot}}$. Clearly $\widetilde{J}$ is an antiunitary on $N(A)^{\bot}$. Since $A^{\#}=-A$, we get that $J=-J^{\#}$ on $R(|A|)$, in turn  on $\overline{R(|A|)}=N(A)^{\bot}$. This means that $\widetilde{J}^{\#}=-\widetilde{J}$. Hence $\widetilde{J}$ is an anticonjugation on $N(A)^{\bot}$. 

We consider the following cases, which exhaust all possibilities;\\
\noindent Case 1: $N(A)={\{0}\}$;\\
In this case, $\widetilde{J}$ is an anticonjugation. So we can take $\kappa:=\widetilde{J}$.
\medskip

\noindent Case 2: $N(A)$ is finite dimensional;\\
In this case $N(A)$ must be even dimensional, say $\dim N(A)=2n$, for some $n\in \mathbb N$. Let ${\{e_1,f_1,e_2,f_2,\dots, e_{n}, f_{n}}\}$ be an orthonormal basis for $N(A)$.  Define an anticonjugation $\kappa_0$ on $N(A)$ by 
\begin{equation*}
\kappa_0(e_j)=f_j\; \text{and}\; \kappa_0(f_j)=-e_j
\end{equation*}
for $j=1,2,\dots, n$. Next, define $\kappa=\kappa_0\oplus \widetilde J$ on $N(A)\oplus N(A)^{\bot}$. Then clearly, $\kappa$ is an anticonjugation satisfying the stated conditions.

\medskip
\noindent Case 3: $N(A)$ is infinite dimensional;\\

Choose an anticonjugation $\kappa_0$ on $N(A)$. Then $\kappa=\kappa_0\oplus \widetilde{J}$ on $N(A)\oplus N(A)^{\bot}$ is an anticonjugation on $H$ and it can be easily seen that  $A=\kappa |A|=|A|\kappa$. 

On the other hand, let $\kappa$ be an anticonjugation satisfying
$\kappa|A|=|A|\kappa$. Since 
$A=\kappa|A| \quad\text{and}\quad \kappa^2=-I,$ we have
$\kappa A=\kappa^2|A|=-|A|$ and 
$A\kappa=\kappa|A|\kappa=\kappa\kappa|A|=-|A|.$ Hence $\kappa A=A\kappa.$ Therefore, if $x\in N(A)$, then

$$
A(\kappa x)=\kappa Ax=0,
$$

and hence $\kappa x\in N(A)$. Thus $N(A)$ is invariant under $\kappa$. Since $\kappa$ is antiunitary, $N(A)^\bot$ is also invariant under $\kappa$. Hence $N(A)$ is a reducing subspace for $\kappa$.

Consequently, the restriction $\kappa_0=\kappa|_{N(A)}$ is an anticonjugation on $N(A)$. By \cite{GarciaTener}[Lemma 4.3], there does not exist an anticonjugation on an odd-dimensional Hilbert space. Therefore, if $N(A)$ is finite dimensional, then $\dim N(A)$ must be even. Thus $\dim N(A)=\infty \quad\text{or}\quad \dim N(A)\ \text{is even}.$ Note that $A^{\#}=(\kappa |A|)^{\#}=|A|\kappa^{\#}=-|A|\kappa=-\kappa |A|=-A$. Hence $A$ is skew-self-adjoint with $\dim N(A)=\infty$ or $\dim N(A)$ even.
\end{proof}
\begin{remark}
Theorem \ref{polardecomp} is analogous to that of \cite[Corollary 3.4]{Szhu2015} proved for bounded linear skew-symmetric operators on a Hilbert space $H$.

\end{remark}

\begin{lemma}\cite[Lemma 3.3]{Santtu}
Let $A$ be a bounded  antilinear operator and $T$ be a self-adjoint complex linear operator with the spectral representation $T=\int \lambda dE(\lambda)$, where $E$ is the spectral measure for $T$. If $AT=TA$, then $E(M)A=AE(M)$ for every Borel subset $M$ of $\sigma(T)$.
\end{lemma}
\begin{lemma}\label{Schattennormest} \cite[Lemma 4.4]{Santtu}.
For a finite rank antilinear operator $A$ of  rank at most $n$, we have $\|A\|_{p}\leq n^{\frac{1}{p}}\|A\|$ \; ($1<p<\infty$).
\end{lemma}
Let \(A\in\mathcal B_a(H)\) be skew-self-adjoint and suppose that \(\dim N(A)\) is either even or infinite. By \cite[Proposition 2.16]{Santtu2}, the spectrum of an antilinear skew-self-adjoint operator is a subset of ${\{0}\}$. Thus, unlike the self-adjoint case, one cannot directly use a spectral measure associated with $A$ in the usual sense. Instead, we shall use the spectral measure of the positive linear operator $|A|=(A^\#A)^{1/2}$.

By Theorem 2.1, there exists an anticonjugation \(\kappa\) such that
$$A=\kappa |A|=|A|\kappa.$$
Let $E$ denote the spectral measure of $|A|$. The commutation relation
$\kappa |A|=|A|\kappa$ 
implies, by Lemma 2.3, that
\begin{equation*}
\kappa E(M)=E(M)\kappa
\end{equation*}
for every Borel subset $M$ of $\sigma(|A|)$. This relation will be used repeatedly in the proof of the following theorem.

\begin{definition}
    Let $A\in \mathcal K_{a}(H)$. Then the singular values $s_n(A) \; (n\in \mathbb{N})$ are the eigenvalues of $|A|=(A^{\#}A)^{\frac{1}{2}}$. We say $A$  is in the Schatten $p$-class $\mathcal K_{a}^{p}(H),\; (1\leq p<\infty)$, if 
 \begin{equation*} 
 \|A\|_{p}=\displaystyle \left (\sum_{n=1}^{\infty}s_n(A)^{p}\right)^{1/p}<\infty.
 \end{equation*}
 \end{definition}
\begin{lemma}\label{orthogonallemma}
Let  $\kappa$ be an anticonjugation on $H$ and $h\in H$ be nonzero. Then $\langle h,\kappa h\rangle=0$.
\end{lemma}
\begin{proof}
By definition of anticonjugation we get $\langle h,\kappa h\rangle=\langle \kappa^2 h, \kappa h\rangle=-\langle h,\kappa h \rangle$. Hence we get the conclusion.
\end{proof}

\begin{theorem}\label{rankprojection}
Let $A\in \mathcal B_{a}(H)$ be skew-self-adjoint with $\dim N(A)=\infty$ or $\dim N(A)$ even and let $\kappa$ be an anticonjugation with $A=\kappa |A|$ and $\kappa |A|=|A|\kappa$ (such a $\kappa$ exists by Theorem \ref{polardecomp}). Let $f\in H$ be nonzero. Then for every $\epsilon>0$ there exists  \begin{enumerate}
\item a finite rank (even rank) orthogonal projection $P$ 
\item an antilinear skew-self-adjoint operator $K\in \mathcal K^{p}_{a}(H)$, $(1<p<\infty)$, with $\|K\|_{p}
<\epsilon$
\end{enumerate}
such that 
\begin{enumerate}
\item $f,\kappa f\in R(P)$
\item $A-K$ is reduced by $R(P)$.
\end{enumerate}
Furthermore, $\kappa P=P\kappa$ and $K\kappa=\kappa K$. Also, $A-K=PAP+(I-P)A(I-P)$.
\end{theorem}
\begin{proof}
Let $|A|=\int_{a}^{b}\lambda\,dE(\lambda)$ be the spectral representation of $|A|$, where $\sigma(|A|)\subseteq [a,b]$ and $E$ is the spectral measure associated to $|A|$.

Fix $n\in\mathbb N$ and put $\delta=\frac{b-a}{n}$. Set $t_j=a+j\delta$, $j=0,1,\ldots,n$, and define $\omega_j=[t_{j-1},t_j)$, $j=1,\ldots,n-1$, and $\omega_n=[t_{n-1},t_n]$. Then $\{\omega_1,\ldots,\omega_n\}$ is a partition of $[a,b]$ and $\operatorname{length}(\omega_j)\leq\delta$ for every $j$.

For each $j$, define
\begin{equation*}
f_j=E(\omega_j)f.
\end{equation*}
Discard those indices $j$ for which $f_j=0$. Label the nonzero vectors by $f_1,f_2,\dots, f_m$. For the remaining indices, define
\begin{equation*}
u_j=\frac{f_j}{\|f_j\|},\qquad v_j=\kappa u_j.
\end{equation*}
Since $\kappa$ commutes with $|A|$, it commutes with every spectral projection $E(\omega_j)$. Hence $u_j,v_j\in E(\omega_j)(H)$. Moreover, $\kappa u_j=v_j$ and $\kappa v_j=-u_j$.

If $j\neq k$, then $u_j\perp u_k$, $v_j\perp v_k$, and $u_j\perp v_k$, because the corresponding vectors belong to mutually orthogonal spectral subspaces of $|A|$. Also, by Lemma \ref{orthogonallemma}, $\langle u_j,v_j\rangle=\langle u_j,\kappa u_j\rangle=0$. Thus $\{u_1,\ldots,u_m,v_1,\ldots,v_m\}$ is an orthonormal set, where $m\leq n$.

Let $M=\operatorname{span}\{u_1,\ldots,u_m,v_1,\ldots,v_m\}$ and let $P$ be the orthogonal projection onto $M$. Then $\operatorname{rank}P=2m\leq2n$. Since $\kappa u_j=v_j$ and $\kappa v_j=-u_j$, the space $M$ is invariant under $\kappa$. Hence
\begin{equation}\label{commutingprop}
\kappa P=P\kappa.
\end{equation}
Furthermore, $f=\sum_{j=1}^m\|f_j\|u_j\in M$ and $\kappa f=\sum_{j=1}^m\|f_j\|v_j\in M$.

Put $Q=I-P$. We next estimate the off-diagonal part of $2\times 2$ operator matrix of $A$ with respect to the decomposition $H=M\oplus M^\perp$. Fix $j$ and choose $\lambda\in\omega_j$. Since $u_j\in E(\omega_j)H$, $Au_j=\kappa|A|u_j$ and $v_j=\kappa u_j$. Therefore
\[
Au_j-\lambda v_j=\kappa(|A|-\lambda I)u_j.
\]
Since the spectrum of $|A|$ restricted to $E(\omega_j)H$ is contained in $\omega_j$,
\[
\|Au_j-\lambda v_j\|=\|(|A|-\lambda I)u_j\|\leq\delta.
\]
As $v_j\in M$, we have $Qv_j=0$, and hence
\[
\|QAu_j\|=\|Q(Au_j-\lambda v_j)\|\leq\delta. \tag{2.5}
\]

Similarly, using $\kappa^2=-I$ and $\kappa|A|=|A|\kappa$, we obtain
\[
Av_j=\kappa |A|\kappa u_j=-|A|u_j.
\]
Consequently,
\[
Av_j+\lambda u_j=-(|A|-\lambda I)u_j,
\]
and therefore
\[
\|QAv_j\|\leq\delta. \tag{2.6}
\]

We now establish the orthogonality needed for the estimate of $QAP$. Since $M=\bigoplus_{j=1}^m\operatorname{span}\{u_j,v_j\}$ and $\operatorname{span}\{u_j,v_j\}\subseteq E(\omega_j)H$, the projection $P$ commutes with every $E(\omega_j)$. Hence
\[
QE(\omega_j)=E(\omega_j)Q. \tag{2.7}
\]
Since $Au_j=\kappa|A|u_j\in E(\omega_j)H$ and $Av_j=-|A|u_j\in E(\omega_j)H$, it follows from (2.7) that $QAu_j,QAv_j\in E(\omega_j)H$. Thus, for $j\neq k$,
\[
QAu_j\perp QAu_k,\qquad QAv_j\perp QAv_k,\qquad QAu_j\perp QAv_k. \tag{2.8}
\]

For the two vectors corresponding to the same index, put $h_j=|A|u_j$. Then $Au_j=\kappa h_j$ and $Av_j=-h_j$. By Lemma \ref{orthogonallemma},
\begin{equation*}
\langle Au_j,Av_j\rangle=-\langle\kappa h_j,h_j\rangle=0.
\end{equation*}
On the other hand, using (\ref{commutingprop}), $PAu_j=P\kappa h_j=\kappa Ph_j$ and $PAv_j=-Ph_j$. Again by Lemma \ref{orthogonallemma},
\[
\langle PAu_j,PAv_j\rangle=-\langle\kappa Ph_j,Ph_j\rangle=0.
\]
Since $P$ is an orthogonal projection,
\[
\langle QAu_j,QAv_j\rangle=\langle Au_j,Av_j\rangle-\langle PAu_j,PAv_j\rangle=0.
\]
Consequently, $\{QAu_j,QAv_j:1\leq j\leq m\}$ is an orthogonal family.

Let $h\in M$. Write
\[
h=\sum_{j=1}^m(\alpha_j u_j+\beta_jv_j).
\]
Using the above orthogonality and (2.5)--(2.6), we obtain
\[
\begin{aligned}
\|QAh\|^2&=\sum_{j=1}^m|\alpha_j|^2\|QAu_j\|^2+\sum_{j=1}^m|\beta_j|^2\|QAv_j\|^2\\
&\leq\delta^2\sum_{j=1}^m(|\alpha_j|^2+|\beta_j|^2)=\delta^2\|h\|^2.
\end{aligned}
\]
Therefore
\[
\|QAP\|\leq\delta=\frac{b-a}{n}. \tag{2.9}
\]

Since $\operatorname{rank}(QAP)\leq2n$, by Lemma \ref{Schattennormest}, we have
\[
\begin{aligned}
\|QAP\|_p&\leq(2n)^{1/p}\|QAP\|\\
&\leq(2n)^{1/p}\frac{b-a}{n}\\
&=2^{1/p}(b-a)n^{-1/q},
\end{aligned}
\tag{2.10}
\]
where $\frac{1}{p}+\frac{1}{q}=1$.

Now define
\[
B=PAP+QAQ\qquad\text{and}\qquad K=QAP+PAQ.
\]
Then $A=B+K$, and hence $A-K=B$. Thus $M=R(P)$ reduces $A-K$.

Moreover,
\[
\begin{aligned}
K^{\#}&=(QAP)^{\#}+(PAQ)^{\#}\\
&=PA^{\#}Q+QA^{\#}P\\
&=-PAQ-QAP\\
&=-K.
\end{aligned}
\]
Therefore $K$ is an antilinear skew-self-adjoint operator. Since $QAP$ has finite rank, $K$ is also finite rank and hence belongs to every Schatten $p$-class, $1<p<\infty$.

Finally,
\[
\begin{aligned}
\|K\|_p&\leq\|QAP\|_p+\|PAQ\|_p\\
&=2\|QAP\|_p\\
&\leq2^{1+1/p}(b-a)n^{-1/q}.
\end{aligned}
\tag{2.11}
\]
Since $n^{-1/q}\to0$ as $n\to\infty$, we may choose $n$ sufficiently large so that $\|K\|_p<\epsilon$.

Finally, since $A\kappa=\kappa A=-|A|$, $P\kappa=\kappa P$, and $Q\kappa=\kappa Q$, we have $K\kappa=\kappa K$. Also, using the fact $A=PAP+PA(I-P)+(I-P)AP+(I-P)A(I-P)$, we get that $A-K=PAP+(I-P)A(I-P)$. This completes the proof.
\end{proof}
\begin{remark}
Theorem \ref{rankprojection} is the skew-self-adjoint counterpart of \cite[Lemma 4.6]{Santtu}: both arguments use a finite partition of the spectrum of $|A|$, construct a finite-rank projection from the corresponding spectral pieces, estimate the off-diagonal corner in the Schatten $p$-norm, and perturb the operator so that the resulting finite-dimensional subspace becomes reducing.
\end{remark}
Next, we describe a relation between the adjoint and transpose of a bounded linear operator.
\begin{lemma}\label{transposeadjoint}
Let ${\{e_n:n\in \mathbb{N}}\}$ be an orthonormal basis for $H$ and $\tau$ be a conjugation on $H$ such that $\tau(e_n)=e_n$ for all $n\in \mathbb{N}$. If $T\in \mathcal B(H)$, then $ T^{tr}=\tau T^* \tau$, where $T^{tr}$ denotes the transpose of $T$ with respect to the basis ${\{e_n:n\in \mathbb{N}}\}$.
    \end{lemma}
\begin{proof}
Let $(t_{i,j})$ be the matrix of $T$ with respect to the basis $B={\{e_n:n\in \mathbb N}\}$, where $t_{i,j}=\langle T(e_j),e_i\rangle$ for all $i,j\in \mathbb{N}$. Then the matrix of $T^*$ with respect to $B$ is $(\overline{t_{j,i}})$. The matrix of  $\tau T^*\tau$ with respect to $B$ is given by $(s_{i,j})$, where $s_{i,j}=\langle \tau T^*\tau (e_j),e_i\rangle$.  That is, 
\begin{align*}
s_{i,j}&=\langle \tau T^*\tau (e_j),e_i\rangle\\
       &= \langle \tau(e_i), T^* \tau (e_j)\rangle\\
       &= \langle e_i, T^*(e_j)\rangle\\
        &= \langle Te_i, e_j\rangle\\
        &=t_{j,i},
        \end{align*}
        the $(j,i)^{th}$ entry of $T$. This completes the proof.
\end{proof}

Here we recall the Takagi factorization theorem for skew-symmetric matrices.
\begin{theorem}\cite[Theorem 1]{standerwiegman}\label{csmdiagonalization}
Let $A$  be an $n\times n$ complex skew-symmetric matrics. Then there exists a  $n\times n$ unitary matrix $U$ 
such that
\begin{equation}\label{csmdiageq}
A = U \left( \begin{bmatrix} 
0& r_1\\
-r_1&0
\end{bmatrix} \oplus \dots \oplus \begin{bmatrix} 
0& r_k\\
-r_k&0
\end{bmatrix} \oplus 0_{n-2k}\right)U^{tr},
\end{equation}
where  $r_1,\dots,r_k>0$.
\end{theorem}
Next, we prove a theorem similar to Takagi factorization theorem for antilinear skew-self-adjoint operators.

\begin{theorem}\label{skewdiagonalization}
Let $A$ be an $n\times n$ antilinear skew-self-adjoint matrix. Then there exists an $n\times n$ unitary matrix $U$ and $r_1,\dots,r_k>0$ such that
\[
A=U\left(r_1\kappa\oplus r_2\kappa\oplus\dots\oplus
r_k\kappa\oplus 0_{n-2k}\right)U^{*},
\]
 $\kappa:\mathbb{C}^2\rightarrow\mathbb{C}^2$
is an anticonjugation given by
\begin{equation*}
\kappa(x,y)=(\bar y,-\bar x),\qquad \text{for all }x,y\in\mathbb{C}.
\end{equation*}
\end{theorem}

\begin{proof}
Let $\tau$ be the conjugation on $\mathbb{C}^n$ defined on the orthonormal
basis $\{e_j:j=1,2,\dots,n\}$ by $\tau(e_j)=e_j$ for all
$j=1,2,\dots,n$. Then $A\tau$ is a complex skew-symmetric linear operator
on $\mathbb{C}^n$. Hence, by Theorem \ref{csmdiagonalization}, there exists
a unitary operator $U$ and $r_m>0$, $m=1,2,\dots,k$, such that
\begin{equation*}
A\tau
=
U(r_1E\oplus r_2E\oplus\dots\oplus r_kE
\oplus 0_{n-2k})U^{tr},
\end{equation*}
where
\[
E=
\begin{pmatrix}
0&1\\
-1&0
\end{pmatrix},
\qquad j=1,2,\dots,k.
\]
Thus
\begin{align*}
A
&=
U(r_1E\oplus r_2E\oplus\dots\oplus r_kE
\oplus 0_{n-2k})U^{tr}\tau\\
&=
U(r_1E\oplus r_2E\oplus\dots\oplus r_kE
\oplus 0_{n-2k})\tau\,\tau U^{tr}\tau.
\end{align*}
Writing
\[
\tau=\tau_1\oplus\tau_2\oplus\dots\oplus\tau_k
\oplus\tau_{n-2k},
\]
where
\[
\tau_j(x,y)=(\bar x,\bar y),\qquad x,y\in\mathbb{C},
\]
for $j=1,2,\dots,k$, and $\tau_{n-2k}$ is the conjugation on
$\mathbb{C}^{n-2k}$, given by $\tau_{n-2k}=\tau|_{\mathbb{C}^{n-2k}}$, we obtain
\begin{equation*}
A=
U(r_1\kappa\oplus r_2\kappa\oplus\dots\oplus r_k\kappa
\oplus 0_{n-2k})U^{*},
\end{equation*}
where
\[
\kappa(x,y)
=
E\tau_j(x,y)
=
(\bar y,-\bar x),
\qquad x,y\in\mathbb{C}.
\]
This completes the proof.
\end{proof}

\begin{remark}
    Let $A\in \mathcal B_{a}(\mathbb C^n)$ be skew-self-adjoint. Then  by Theorem \ref{skewdiagonalization}, we can conclude that the rank of $A$ is even. If the rank of $A=2k$, for some $k\in \mathbb{N}$, 
     then there exists  real numbers $r_j> 0$ for $j=1,2,\dots k$ and an orthonormal basis ${\{e_j,f_j:1\leq j\leq k}\}$ such that   \begin{equation}
    Ax=\sum_{j=1}^{k}r_j \left(\overline{\langle x, e_j\rangle} f_j-\overline{\langle x, f_j\rangle} e_j\rangle \right), \text{for all}\; x\in \mathbb{C}^n.
    \end{equation}
That is, $Ae_j=r_jf_j$ and $Af_j=-r_je_j$ for all $j=1,2,\dots,k$.
\end{remark}

In order to prove Theorem \ref{WvN}, we first establish several lemmas that will simplify the proof.
\begin{lemma}\label{lem:A}
Let $\kappa$ be an anticonjugation on a Hilbert space $H$, let $S\ge0$ be a bounded linear operator with
$\kappa S=S\kappa$, and put $A=\kappa S$. Let $P$ be an orthogonal projection with $P\kappa=\kappa P$, put $Q=I-P$ and
$B=PAP+QAQ$. Then
\begin{enumerate}
\item $B=\kappa S',\; \text{with}\; S'=PSP+QSQ \geq 0$,
\item $\kappa S'=S'\kappa$, and
\item $ S'=|B|$.
\end{enumerate}
\end{lemma}
\begin{proof}
Since $P\kappa=\kappa P$, $PAP=P\kappa SP=\kappa PSP$, and likewise $QAQ=\kappa QSQ$; hence $B=\kappa S'$.
Clearly $S'\ge0$. As $\kappa$ commutes with $P,Q,S$, we get $\kappa PSP=PS\kappa P=PSP\kappa$, and similarly for $QSQ$; so $\kappa S'=S'\kappa$.
Finally $\kappa^{\#}\kappa=I$, so $B^{\#}B=S'\kappa^{\#}\kappa S'=S'^2$, i.e.\ $|B|=S'$.
\end{proof}

\begin{lemma}\label{lem:B}
Let $\kappa$ be an anticonjugation on $H$ and $A\in\mathcal B_a(H)$ with $A=\kappa|A|$ and $\kappa|A|=|A|\kappa$.
Let $P$ be a finite rank orthogonal projection that commutes with $A$ and with $\kappa$, and put $H_P=(I-P)(H)$, $\kappa_P=\kappa|_{H_P}$, $\widetilde A=A|_{H_P}$. Then $\kappa_P$ is an anticonjugation on $H_P$, and
\[
\widetilde A=\kappa_P|\widetilde A|,\; \text{and}\; \kappa_P|\widetilde A|=|\widetilde A|\kappa_P .
\]
\end{lemma}
\begin{proof}
Write $A=\kappa S$ with $S=|A|$. Since $P$ commutes with $A$ and $A^{\#}=-A$, $P$ commutes with $A^{\#}A=S^2$, hence with $S$ (functional calculus applied to the
positive operator $S^2$). Thus $H_P$ reduces $S$ and $\kappa$ (the latter by hypothesis), hence it reduces $A=\kappa S$. So $\kappa_P$ is a well-defined
anticonjugation on $H_P$, and, writing $\widetilde S=S|_{H_P}\ge0$,
\[
\widetilde A=A|_{H_P}=\kappa_P\widetilde S,\;\text{and} \;\kappa_P\widetilde S=\widetilde S\kappa_P
\]
(the second identity follows since $\kappa$ commutes with $S$ on all of $H$). Since $\kappa_P$ is antiunitary, $\kappa_P^{\#}\kappa_P=I_{P}$ (the identity operator on $H_{P}$), so
\[
\widetilde A^{\#}\widetilde A=(\kappa_P\widetilde S)^{\#}(\kappa_P\widetilde S)=\widetilde S\kappa_P^{\#}\kappa_P\widetilde S=\widetilde S^2 .
\]
As $\widetilde S\ge0$, this gives $\widetilde S=(\widetilde A^{\#}\widetilde A)^{1/2}=|\widetilde A|$, which proves the lemma.
\end{proof}

\begin{lemma}[Induction step]\label{inductionlemma}
Let $\kappa$ be an anticonjugation on $H$ and $A\in\mathcal B_a(H)$ with $A=\kappa|A|$, and $\kappa|A|=|A|\kappa$.
Fix $\epsilon>0$, an orthonormal set $\{e_n:n\in\mathbb N\}$ such that $\{e_n,\kappa e_n:n\in\mathbb N\}$ is an orthonormal basis of $H$. Suppose that finite even rank orthogonal projections $P_1,\dots,P_n$ and finite rank antilinear skew-self-adjoint $K_1,\dots,K_n\in\mathcal K_a^p(H)$
satisfy, with $P^{(n)}=P_1+\dots+P_n$, $A^{(n)}=A-K_1-\dots-K_n$ $($and $P^{(0)}=0$, $A^{(0)}=A)$:
\begin{enumerate}
\item \label{smallnorm} $\|K_j\|_p<\epsilon/2^{j}$ for $j\le n$;
\item \label{mutualogp} $P_iP_j=0$ for $i\neq j\le n$, and $P_j\kappa=\kappa P_j$ for $j\le n$;
\item \label{dualvectors} $e_j,\kappa e_j\in R(P^{(n)})$ for $j=1,\dots,n$;
\item \label{reducedpolarform} $A^{(n)}=\kappa|A^{(n)}|$ and $\kappa|A^{(n)}|=|A^{(n)}|\kappa$;
\item \label{commutingproj} each $P_j$, $1\le j\le n$, reduces $A^{(n)}$;
\item\label{antisskewselfadjprop}  $K_j$ vanishes on $R(P^{(j-1)})$ and $R(K_j)\subseteq R(P^{(j-1)})^{\perp}$, for $j\le n$; $K_j\kappa=\kappa K_j$.
\end{enumerate}
Then there exist an even rank orthogonal projection $P_{n+1}$ and a finite rank antilinear skew-self-adjoint $K_{n+1}\in\mathcal K_a^p(H)$ such that
(1)--(6) hold with $n$ replaced by $n+1$.
\end{lemma}

\begin{proof}
Put $H_n=R(P^{(n)})^{\perp}$. By (\ref{commutingproj}) and (\ref{mutualogp}), $P^{(n)}$ commutes with $A^{(n)}$ and with $\kappa$, so Lemma~\ref{lem:B} applies (with $P=P^{(n)}$):
$H_n$ reduces $A^{(n)}$, $|A^{(n)}|$, $\kappa$, and, writing $\kappa_n=\kappa|_{H_n}$, $\widetilde A_n=A^{(n)}|_{H_n}$, we obtain
\[
\widetilde A_n=\kappa_n|\widetilde A_n|,\qquad \kappa_n|\widetilde A_n|=|\widetilde A_n|\kappa_n .
\]

Since $R(P^{(n)})$ is finite dimensional, choose the least $r$ with $e_r\notin R(P^{(n)})$ and define
$$f_n=\frac{(I-P^{(n)})e_r}{\|(I-P^{(n)})e_r\|}\in H_n.$$ As $P^{(n)}\kappa=\kappa P^{(n)}$, also $\kappa f_n=(I-P^{(n)})\kappa e_r/\|(I-P^{(n)})e_r\|\in H_n$.

Next applying Theorem \ref{rankprojection} to $\widetilde A_n$ with anticonjugation $\kappa_n$, the vector $f_n$, and $\eta=\epsilon/2^{n+1}$, we get  
\begin{enumerate}
\item an even rank orthogonal projection $P_{n+1}$ on $H_n$ and 
\item finite rank $\widetilde K\in\mathcal K_a^p(H_n)$ with $\widetilde K^{\#}=-\widetilde K$,
$\|\widetilde K\|_p<\epsilon/2^{n+1}$,
\end{enumerate}
such that 
\begin{enumerate}
\item $f_n,\kappa_nf_n\in R(P_{n+1})$, $\kappa_nP_{n+1}=P_{n+1}\kappa_n$, $\widetilde K\kappa_n=\kappa_n\widetilde K$, and
\item $\widetilde A_n-\widetilde K=P_{n+1}\widetilde A_nP_{n+1}+Q_{n+1}\widetilde A_nQ_{n+1},\qquad Q_{n+1}=I_{H_n}-P_{n+1}$.
\end{enumerate}
Extend $P_{n+1}$ and $K_{n+1}:=\widetilde K$ by zero on $R(P^{(n)})$, as operators on $H$.

Next we prove the induction step. 

\begin{itemize}
\item[(a)] It clear that (\ref{smallnorm}) of  Lemma \ref{inductionlemma} follows by from the definition of $K_{n+1}$. 

\item[(b)] Induction step for  (\ref{mutualogp}): $R(P_{n+1})\subseteq H_n$ gives $P_{n+1}P_j=0$ for $j\le n$, and
$P_{n+1}\kappa=\kappa P_{n+1}$ since $H_n$ reduces $\kappa$.

\item[(c)] Induction step for (\ref{dualvectors}): By minimality of $r$ (the least index with
$e_r \notin R(P^{(n)})$), we have $e_j \in R(P^{(n)})$ for all $j < r$;
since $R(P^{(n)})$ is $\kappa$-invariant, it follows that
$\kappa e_j \in R(P^{(n)})$ for $j < r$ as well. Combined with (\ref{dualvectors}) for $n$, which places $e_1, \dots, e_n$ in $R(P^{(n)})$, minimality of $r$
forces $r \ge n+1$. Also
$$e_r=P^{(n)}e_r+\|(I-P^{(n)})e_r\|f_n\in R(P^{(n+1)}),$$ and since $\|(I-P^{(n)})e_r\|$ is real, 
$$\kappa e_r=P^{(n)}\kappa e_r+\|(I-P^{(n)})e_r\|\kappa f_n\in R(P^{(n+1)}).$$
If $r=n+1$ this gives (\ref{dualvectors}) for $n+1$; if $r>n+1$ then $e_{n+1},\kappa e_{n+1}\in R(P^{(n)})\subseteq R(P^{(n+1)})$ already.

\item[(d)] Induction step for  (\ref{reducedpolarform}): By Lemma~\ref{lem:A} (on $H_n$, with $\kappa_n$, $S=|\widetilde A_n|$, $P=P_{n+1}$, we get
\[
\widetilde A_n-\widetilde K=\kappa_nS',\; \text{with}\; \ S'=|\widetilde A_n-\widetilde K|\ge0,\;\text{and}\; \kappa_nS'=S'\kappa_n .
\]
On $R(P^{(n)})$, (\ref{reducedpolarform}) for $n$ gives $A^{(n)}|_{R(P^{(n)})}=\kappa S_1$ with $$S_1=|A^{(n)}|\,\big|_{R(P^{(n)})}\ge0,$$ $\kappa S_1=S_1\kappa$ (this subspace reduces $\kappa$ and $|A^{(n)}|$).
So $A^{(n+1)}=\kappa S_{n+1}$ with $S_{n+1}=S_1\oplus S'\ge0$, $\kappa S_{n+1}=S_{n+1}\kappa$; since $\kappa^{\#}\kappa=I$, we have $(A^{(n+1)})^{\#}A^{(n+1)}=S_{n+1}^2$, i.e.\ $S_{n+1}=|A^{(n+1)}|$.

\item[(e)] Induction step for  (\ref{commutingproj}): with respect to $H=R(P^{(n)})\oplus H_n$, we have
$A^{(n+1)}=A^{(n)}|_{R(P^{(n)})}\oplus(\widetilde A_n-\widetilde K)$. Each $P_j$, $j\le n$, acts on on the first summand and reduces $A^{(n)}$ by (\ref{commutingproj}), hence reduces
$A^{(n)}|_{R(P^{(n)})}$; $P_{n+1}$ lies in the second summand and commutes with $$\widetilde A_n-\widetilde K=P_{n+1}\widetilde A_nP_{n+1}+Q_{n+1}\widetilde A_nQ_{n+1}.$$

\item[(f)] Induction step for (\ref{antisskewselfadjprop}): by construction $K_{n+1}$ vanishes on $R(P^{(n)})$, so we have $R(K_{n+1})\subseteq H_n=R(P^{(n)})^{\perp}$, and $K_{n+1}\kappa=\kappa K_{n+1}$ since $H_n$ reduces $\kappa$ and $\widetilde K\kappa_n=\kappa_n\widetilde K$. 
\end{itemize}
\end{proof}

\medskip
\noindent \textbf{Proof of Theorem \ref{WvN}}:
\begin{proof}
Let $A=\kappa|A|$ be the decomposition of $A$ as in Theorem \ref{polardecomp}. Let $\{e_n:n\in\mathbb N\}$ be an orthonormal set such that $\{e_n,\kappa e_n:n\in\mathbb N\}$ is an orthonormal basis of $H$.

Starting from $n=0$ (where (\ref{smallnorm})--(\ref{antisskewselfadjprop}) hold vacuously except (\ref{mutualogp}), which holds since $A^{(0)}=A=\kappa|A|$), apply Lemma~\ref{inductionlemma} repeatedly to obtain
sequences $(P_n)$, $(K_n)$ satisfying \ref{smallnorm})--(\ref{antisskewselfadjprop}) for every $n$.

 Define
\[
K=\sum_{j=1}^{\infty}K_j\qquad\text{and}\qquad D=A-K .
\]
Since $\sum_j\|K_j\|_p<\epsilon$, the series converges in the Schatten $p$-norm (hence in operator norm), $K\in\mathcal K_a^p(H)$, $K^{\#}=-K$, and
$\|K\|_p\le\sum_j\|K_j\|_p<\epsilon$. Thus $A=K+D$, and $D$ is antilinear skew-self-adjoint.

We claim that $P_n$ commutes with $D$. To this end, fix $n$ and  define
\begin{equation*}
D=A^{(n)}-\sum_{j>n}K_j\quad ( \text{convergence in operator norm}). 
\end{equation*}

For $j>n$, $R(P_n)\subseteq R(P^{(j-1)})$, so $K_jP_n=0$ by (\ref{antisskewselfadjprop}), and $R(K_j)\subseteq R(P^{(j-1)})^{\perp}\perp R(P_n)$, so $P_nK_j=0$. By (\ref{commutingproj}), $P_n$ commutes with $A^{(n)}$. Hence $P_nD=DP_n$.

\emph{$H=\bigoplus_nR(P_n)$.} The ranges $R(P_n)$ are mutually orthogonal by (\ref{mutualogp}); by (\ref{dualvectors}), every $e_j,\kappa e_j$ lies in some $R(P^{(m)})$; since $\{e_j,\kappa e_j\}$
is an orthonormal basis of $H$, $\sum_nP_n=I$ strongly. With $D_n=D|_{R(P_n)}$,
\[
D=\bigoplus_{n=1}^{\infty}D_n .
\]
Note that $R(P_n)$ is finite dimensional of even dimension, and $D_n$ is antilinear skew-self-adjoint on it. By Theorem \ref{skewdiagonalization} and
Remark~2.12, $D_n$ has even rank, so $\dim N(D_n)$ is even, and $D_n$ is a direct sum of $2\times2$ skew blocks $d\kappa$ with $d\ge0$ (the kernel splitting into blocks with $d=0$).
Collecting the orthonormal bases of all $R(P_n)$ gives an orthonormal basis $\{e_n',f_n':n\in\mathbb N\}$ of $H$ and numbers $d_n\ge0$ with $De_n'=d_nf_n'$, $Df_n'=-d_ne_n'$;
$D$ is bounded since $A$ and $K$ are. Hence $D$ is a block skew-diagonal operator, $A=K+D$, and $\|K\|_p<\epsilon$.
\end{proof}

\begin{remark}\label{skewdiagonalrmk}
\begin{enumerate}
\item The operator $D$ can be defined by 
\begin{equation}\label{antilinearskewdiagonal}
Dx=\displaystyle \sum_{n=1}^{\infty}d_n \left(\overline{\langle x,e_n\rangle} f_n-\overline{\langle x,f_n\rangle} e_n \right),\; \text{for all}\; x\in H.
\end{equation}
In fact, we can write $D=\oplus_{n=1}^{\infty}d_n\kappa$, where $d_n\geq 0$ and 
$\kappa:\mathbb{C}^2\rightarrow \mathbb{C}^2$ is an anticonjugation given by
$\kappa(x,y)=(-\bar{y},\bar{x}),\text{for all}\; x,y\in \mathbb{C}$.
\item Note that $De_n=d_nf_n$ and $Df_n=-d_ne_n$ for all $n\in \mathbb N$. That is,  with respect to the orthonormal basis ${\{e_n,f_n:n\in \mathbb{N}}\}$, the matrix of $D$ is nothing but $D=\displaystyle \bigoplus_{n=1}^{\infty} \begin{pmatrix}
    0&d_n\\
    -d_n&0
\end{pmatrix}$.
\item It is to be noted that the iterative proof of Theorem \ref{WvN} closely parallels the proof of \cite[Theorem 4.5]{Santtu}, where the finite-dimensional reduction lemma is applied successively to obtain an orthogonal decomposition into finite-dimensional reducing subspaces. The principal additional feature in the present skew-self-adjoint setting is that the spectral pieces are treated in $\kappa$-paired form and the resulting finite-dimensional restrictions are reduced to $2\times 2$ skew blocks instead of one-dimensional diagonal blocks.
\end{enumerate}
\end{remark}
\subsection{Skew-symmetric Operators}
Next, our aim is to establish the Weyl--von Neumann theorem for complex skew-symmetric operators. This can be done by connecting the antilinear skew-self-adjoint operators and complex skew-symmetric operators.  

\medskip
\noindent{\textbf{Proof of Theorem \ref{skewsymWvN}}}:

\begin{proof}
Let $\{e_n:n\in\mathbb N\}$ be an orthonormal basis of $H$ with
$\tau e_n=e_n$ for all $n$, and let $tr$ denote the transpose with respect
to this basis, so that $X^{tr}=\tau X^{*}\tau$ by Lemma~\ref{transposeadjoint}.

 Define $A:=T\tau$. Since $T$ is $\tau$-skew-symmetric, i.e.\ $\tau T^{*}\tau=-T$,
\[
A^{\#}=(T\tau)^{\#}=\tau T^{*}=-T\tau=-A ,
\]
so $A$ is an antilinear skew-self-adjoint operator. Moreover
$N(A)=\tau N(T)$, hence $\dim N(A)=\dim N(T)$ is even or infinite dimensional.
By Theorem~ \ref{WvN}, there exist an antilinear skew-self-adjoint Schatten $p$-class
operator $K_0$ with $\|K_0\|_p<\epsilon$ and an antilinear block skew-diagonal
operator $B$ such that
\[
A=K_0+B .
\]
Explicitly, there are $d_n\ge 0$ and an orthonormal basis
$\{f_n,g_n:n\in\mathbb N\}$ of $H$ with
\begin{equation*}
Bf_n=d_n g_n,\;\text{and}\; Bg_n=-d_n f_n\; \text{for all}\; n\in\mathbb N.
\end{equation*}

Next, define a linear operator $U:H\rightarrow H$ by
\begin{equation*}
Ue_{2n-1}=g_n,\;\text{and}\; Ue_{2n}=f_n\; \text{for all}\; n\in\mathbb N.
\end{equation*}
Since $\{f_n,g_n\}$ is an orthonormal basis for $H$, $U$ is unitary.
Let $B_0:=U^{*}BU$. Then $B_0$ is antilinear, $B=UB_0U^{*}$, and
\begin{align*}
B_0e_{2n-1}&=U^{*}Bg_n=-d_nU^{*}f_n=-d_n e_{2n},\\
B_0e_{2n}&=U^{*}Bf_n=d_nU^{*}g_n=d_n e_{2n-1}.
\end{align*}
Also $B_0^{\#}=U^{*}B^{\#}U=-B_0$, so $B_0$ is antilinear skew-self-adjoint.

\medskip
Define  $D:=B_0\tau$. Since $\tau e_k=e_k$ for all $k$,
$D$ and $B_0$ agree on the basis $\{e_k\}$, we have

\begin{equation*}
De_{2n-1}=-d_n e_{2n},\;\text{and}\; De_{2n}=d_n e_{2n-1},\; \text{for all}\; n\in \mathbb{N}.
\end{equation*}
i.e.\ $D=\bigoplus_{n=1}^{\infty} d_n C$ with
$C=\begin{pmatrix}0&1\\-1&0\end{pmatrix}$. Since $D\tau=B_0\tau\tau=B_0$ is antilinear
skew-self-adjoint, $D$ is $\tau$-skew-symmetric.

Since $\tau^2=I$ and $U^{*}\tau=\tau(\tau U^{*}\tau)=\tau U^{tr}$, we obtain
\[
B\tau=UB_0U^{*}\tau=UB_0\,\tau\,U^{tr}=U(B_0\tau)U^{tr}=UDU^{tr}.
\]
Multiplying $A=K_0+B$ on the right by $\tau$ gives
\[
T=A\tau=K_0\tau+B\tau=K+UDU^{tr},
\]
where $K:=K_0\tau$.

Next we verify the properties of $K$. It is clear that $K$ is linear, and $K\tau=K_0\tau\tau=K_0$ is antilinear skew-self-adjoint, so
$K$ is $\tau$-skew-symmetric. Further, $K^{*}=\tau K_0^{\#}$, hence
\[
K^{*}K=\tau K_0^{\#}K_0\tau=\tau|K_0|^{2}\tau .
\]
Since $\tau$ is an antiunitary involution, $\tau|K_0|^{2}\tau$ is a positive
operator with the same eigenvalues as $|K_0|^{2}$, so $K$ and $K_0$ have the
same singular values. Therefore $K$ is in the Schatten $p$-class and
\[
\|K\|_p=\|K_0\|_p<\epsilon .
\]
This completes the proof.
\end{proof}

In the next result, we drop the condition on the null space of the operator. We give another sufficient condition so that the conclusion of Theorem \ref{skewsymWvN} holds true.
\begin{corollary}\label{indpendantssop}
   Let $\tau$ be a conjugation on $H$ and let $T$ be a $\tau$-skew-symmetric operator such that $N(T^*)=N(T)$. Then for every $\epsilon>0$, there exists 
   \begin{enumerate}
   \item a $\tau$-skew-symmetric Schatten $p$-class operator $K$, 
   \item a block  skew-symmetric operator $D$  and a unitary operator $U$ 
   \end{enumerate}
   such that 
   \begin{equation*}
   T=K+UDU^{tr}\;  \text{and}\; \|K\|_{p}<\epsilon,\; (1<p<\infty).
   \end{equation*}
   Here $U^{tr}$ denotes the transpose of the operator $U$ with respect to an orthonormal basis ${\{e_n:n\in \mathbb{N}}\}$ of $H$ such that $\tau(e_n)=e_n$ for all $n\in \mathbb{N}$. 
\end{corollary}
\begin{proof}
Since $N(T)=N(T^*)$ and $N(T)$ is invariant under both $T$ and $T^*$, hence it reduces $T$. So we can write $T=0\oplus T_2$, where $T_2=T|_{N(T)^{\bot}}$. Note that $T_2$ is one-to-one. 

Let $\tau=\begin{pmatrix}
    \tau_{11} & \tau_{12}\\
    \tau_{21} & \tau_{22}
\end{pmatrix}$, be the operator matrix of $\tau$ with respect to the decomposition $H=N(T)\oplus N(T)^{\bot}$. Since $\tau^{\#}=\tau$, we get that $\tau_{jj}^{\#}=\tau_{jj}$ for $j=1,2$. Also we have $\tau_{21}=\tau_{12}^{\#}$. Since $\tau T^*\tau=-T$, it follows that $N(T)$ is invariant under $\tau$. That is, $\tau_{21}=0$. Also, $\tau_{12}=0$. Hence $N(T)$ reduces $\tau$. So we have that $\tau=\tau_{11}\oplus \tau_{22}$. The equation $\tau T^*\tau=-T$ implies that $\tau_{22}T_2^*\tau_{22}=-T_2$. By applying Theorem \ref{skewsymWvN}, we get that for each $\epsilon>0$, there exists $\tau_{22}$-skew-symmetric compact operator $K_2$ with $\|K_2\|_{p}<\epsilon$ and a unitary $U_2$ on $N(T)^{\bot}$ such that 
\begin{equation*}
T_2=K_2+U_2D_2U_2^{tr},
\end{equation*}
where $D_2$ is a block skew-diagonal operator.  Since $T=0\oplus T_2$, we can take 
\begin{equation*}
K=\begin{pmatrix}
    0&0\\
    0&K_2
\end{pmatrix}, \;D=\begin{pmatrix}
    0&0\\
    0&D_2
\end{pmatrix}\; \text{and}\;U=\begin{pmatrix}
    I_1&0\\
    0&U_2
\end{pmatrix},
\end{equation*}
where $I_1$ is the identity operator on $N(T)$.  With these, we get that $T=K+UDU^{tr}$.  Also, it is straight forward to verify that $K$ is $\tau$-skew-symmetric, $\|K\|_p<\epsilon$ and $D$ is  block skew-symmetric. 
\end{proof}

A Weyl--von Neumann type theorem for operators in an orthogonal Lie algebra was proved by Q. Bu and S. Zhu in \cite[Theorem 1.1 (i)]{SZhu} . Here we state the result. 
\begin{theorem*}\label{SZhu} 
Let $T\in \mathcal B(H)$. Let $C$ be a conjugation on $H$ such that $T^*=T=-CTC$ and $1<p<\infty$. Then for every $\epsilon>0$, there exists a self-adjoint diagonal operator $D$ on $H$ with $CDC=-D$ such that $\|T-D\|_{p}<\epsilon$. 
\end{theorem*}
This result is different from Corollary \ref{indpendantssop}, which concerns arbitrary $\tau$-skew-symmetric operators satisfying $N(T)=N(T^*)$ and gives a congruence representation by a block skew-symmetric operator.
\begin{center}
	\textbf{Ethics declarations}
\end{center}

\noindent \textbf{Conflict of Interest}\\
The author declares that there are no competing interests relevant to the content of this article.
\medskip

\noindent \textbf{Ethical Approval}\\
Ethics approval is not applicable to this article, as no research involving human or animal subjects was performed.
\section*{Acknowledgements}
We sincerely thank the referee for the careful reading of the manuscript. The referee's insightful questions led to significant improvements in the proofs of Theorems \ref{WvN} and \ref{rankprojection}, and their suggestions also helped improve the overall clarity of the paper. We also thank the handling editor for their support during the review process.

\end{document}